\documentclass[a4paper,11pt]{article}
\usepackage{a4wide}
\usepackage{amsmath, amssymb, amsthm}
\usepackage{hyperref}
\usepackage{graphicx}
\usepackage{xcolor}
\usepackage{listings}
\usepackage{float}
\usepackage{caption}
\usepackage{subcaption}

\newtheorem{definition}{Definition}[section]
\newtheorem{theorem}{Theorem}[section]
\newtheorem{lemma}{Lemma}[section]
\newtheorem{conjecture}{Conjecture}[section]
\newcommand{\cont}[2]{{\mathcal #1}.\text{cont}\,{#2}} %$\cont{F}{v}$
\newcommand{\del}[2]{{\mathcal #1}.{\rm del}\,{#2}} 
\newcommand{\delp}[2]{{\mathcal #1}.{\rm del}{'}\,{#2}}
\newcommand{\trace}[2]{{\mathcal #1}.{\rm trace}\,{#2}} 

\title{On the Averaging Problem of Ideal Families Related to Frankl's Conjecture with Formal Proof by Lean 4}
\author{Masahiro Hachimori \quad University of Tsukuba\\
Kenji Kashiwabara\footnote{Corresponding author} \quad  The University of Tokyo}
\date{}

\begin{document}

%\sloppy

\maketitle

\begin{abstract}
Frankl's conjecture, also known as the union-closed sets conjecture, can be equivalently expressed in terms of intersection-closed set families by considering the complements of sets. It posits that any family of sets closed under intersections, and containing both the ground set and the empty set, must have a ``rare vertex'' — a vertex belonging to at most half of the members of the family. 
The concept of \emph{average rarity} describes a set family where the average degree of all the elements is at most half of the number of its members. 
Average rarity is a stronger property that implies the existence of a rare vertex.
This paper focuses on ideal families, which are set families that are downward-closed (except the ground set) and include the ground set. We present a proof that the normalized degree sum of any ideal family is non-positive, which is equivalent to saying that every ideal family satisfies the average rarity condition. 
%Additionally, we formalize and verify the proof using the Lean 4 theorem prover.
This proof is formalized and verified using the Lean 4 theorem prover.
%Our contributions include the presentation of detailed proofs, and the implementation of the proof in Lean 4.
\end{abstract}

\section{Introduction}

% \subsection{Frankl's Conjecture}
\subsection{Frankl's conjecture}

Frankl’s conjecture, also referred to as the union-closed sets conjecture, is a well-known and enduring open problem in combinatorics, first proposed in 1979 by P. Frankl \cite{Frankl1979}. 
% The conjecture has sparked significant interest due to its deceptively simple formulation and far-reaching implications.
Its deceptively simple formulation and far-reaching implications have sparked significant interest.
It states that in any nonempty union-closed family on finite sets, there exists an element that is contained in at least half of the sets.  When considering the dual notion through complements, the conjecture can also be phrased in terms of intersection-closed families, asserting the existence of an element contained in at most half of the members.

Consider set families on a nonempty finite ground set $U$. An element in $U$ is called a \emph{vertex}. We call a member of a set family $\mathcal{F}$ a \emph{hyperedge}.
A family of sets $\mathcal{F}$ is called \emph{intersection-closed} if for any two sets $A, B \in \mathcal{F}$, their intersection $A \cap B$ is also in $\mathcal{F}$. %We consider the case where the ground set $U$ is itself a hyperedge in $\mathcal{F}$. 

In the intersection-closed form, Frankl's conjecture asserts that any nonempty intersection-closed family on a finite set, containing both the empty set and the ground set, must have at least one \emph{rare vertex}, which is an element appearing in at most half of the hyperedges in the family. Instead of assuming the existence of the ground set and the empty set, it is equivalent to assume that the set family contains at least two sets for considering the conjecture. 
% It follows from the fact that we can make a counterexample of the conjecture assuming at least two sets to a counterexample of the conjecture having the ground set as a hyperedge. 

% \subsection{History of the Conjecture and Related Studies}
\bigskip

Despite substantial efforts, the conjecture remains unresolved \cite{Bruhn2015}. However, it has been verified for specific cases, such as families with small cardinality (up to 11 elements \cite{Bosnjak2008}) or those exhibiting particular structural properties \cite{MoghaddasMehr2023,Morris2006,Roberts2010}. 
 It also connects to lattice theory and related areas of combinatorics, highlighting its mathematical richness 
 % \cite{Abe2000,Gilmer2022,nagy2009survey}. 
\cite{Abe2000,Poonen1992,Reinhold2000}. 
 %Computational advances have allowed for the verification of additional cases and refined insights into its combinatorial properties \cite{gowers2008union}. 
 %Recent advancements in computational approaches and probabilistic methods \cite{robbins2022union} have led to partial progress, further deepening our understanding of the conjecture.  
Recent significant progress by Gilmer~\cite{Gilmer2022} establishes that the existence
of an element contained in a constant ratio of sets of the family (for the union-closed form of the conjecture).
The ratio was improved to $\frac{3-\sqrt{5}}{2}\approx 0.381966$ 
by several authors~\cite{add4-AHS,add3-CL,add5-Pebody,add2-Sawin},
and further improvements have been made in subsequent works~\cite{add7-Cambie,add8-Liu,add6-Yu}, but there still remains a gap to the conjectured ratio of 1/2.

Hachimori and Kashiwabara \cite{Hachimori2024} investigated Frankl’s conjecture in the intersection-closed form through minor operations such as deletion and contraction. 
Their work focuses on the properties that minimal counterexamples of the conjecture have to satisfy. 
In this paper, we continue to adopt the approach of the intersection-closed form, but focus on a different direction.
% Their work identifies necessary conditions for the conjecture to hold, contributing to the understanding of structural properties critical to the problem.

In this paper, as explained later, we provide a formalized proof by Lean 4. Previously,
Marić, Zivković, and Vučković \cite{Maric2012}  utilized the Coq proof assistant to formalize FC-families and upward-closed families, offering a structured framework for systematically analyzing these families. 
Coq is another theorem prover system that is widely used.
Their Coq-based techniques are invaluable for handling upward-closed families in formal proof contexts.

%Specifically, it focuses on Tutte's theorem, which states that a matroid is representable over GF(2) if and only if it does not contain $U_{2,4}$$ as a minor. This work leverages Lean 4's theorem proving capabilities to rigorously verify complex mathematical structures in matroid theory. 

 %By building on recent progress , we seek to further elucidate the structure of union-closed families and their intriguing properties.
 %This paper explores the conjecture through new theoretical insights and computational verifications, with the aim of contributing to the resolution of this enduring challenge.

%\subsection{Averaging Approach to Frankl's Conjecture}
\subsection{Averaging approach to Frankl's conjecture}

% In combinatorics, the study of set families and their properties is a fundamental topic. One such property is the condition of \emph{average rarity}, which relates to the distribution of vertices across hyperedges in a set family \cite{Andrey2021, Czedli2009}. 
In this paper, we focus on the averaging approach, which discusses the average degree over all vertices instead of the existence of the rare (or abundant) vertices.
This is one of the important approaches, see \cite[Sec. 6]{Bruhn2015},\cite{Czedli2009} for example.
% This result derives from the double counting principle, which asserts that the total size of all elements in a set family is equal to the total degree sum of all vertices.
%The average rarity condition has implications for understanding the structure and behavior of set families. This property provides insights into how certain combinatorial structures distribute the membership of elements. 
In this paper, the average rarity of set families is the main topic,  where a set family is average rare if the average of the degrees of the vertices is at most 1/2. This is a stronger property than the existence of a rare vertex.
To analyze the average rarity of a set family, we introduce the normalized degree sum of a set family as a measure.
% We define its normalized degree sum as twice the total size of its hyperedges minus the product of the number of vertices and the size of the ground set. 
% If the normalized degree sum of a set family is non-positive, the family is said to satisfy the average rarity condition. A set family satisfying the average rarity condition clearly contains a rare vertex. It is known that there exist set families that do not satisfy the average rarity.
If the normalized degree sum of a set family is non-positive, it is an average rare family.

%The average rarity condition sheds light on the inherent properties of intersection-closed families and their potential applications in combinatorial designs and optimization problems.

It should be remarked that the average rarity is a condition strictly stronger than the existence of a rare vertex, and the intersection-closed families that are average rare form a strict subclass. 
Therefore, the problem to identify the class of average rare families is a different problem from Frankl's conjecture. 
Still, this problem itself has an importance in the study of the combinatorial structures of set families, and may contribute to the conjecture to some extent.
The main theorem of this paper provides a new class of intersection-closed families, ideal families, that are average rare.

% \subsection{Ideal Families}
\bigskip

Ideal families, introduced in Section~\ref{subsec:idealfamily} of this paper, constitute a particular class of set families that are downward-closed except the ground set and include both the empty set and the ground set. Although we have not found the exact same statement, it is probably known that any ideal family has a rare vertex. For example, \cite{Nagel2022} refers that a filter family cannot be a counterexample of the union-closed set conjecture. However, it is different from ours because our ideal family has always the ground set as its hyperedge. 
%Therefore, the class of ideal families does not provide a counterexample to Frankl's conjecture. Firstly, we show the fact in Lemma \ref{lemma:idealrare}.
%Moreover, we show that they satisfy the average rarity condition (Theorem \ref{theorem:main}) by showing that their normalized degree sum is non-positive. This is our main result. 
We show in Lemma~\ref{lemma:idealrare} that ideal families contain rare vertices, that is, no counterexamples to Frankl's conjecture can be found among ideal families. Further, we show in Theorem~\ref{theorem:main} that they are average rare.

% Establishing average rarity for ideal families provides a step toward understanding broader conjectures like Frankl's and how such properties might generalize to other types of set families.

%Although this does not directly prove Frankl's conjecture for ideal families, it narrows the conditions under which such conjectures might hold. In this way, our work can be seen as contributing to the foundational understanding needed to address Frankl's conjecture more fully.

% \subsection{Significance of Formal Approach}
\subsection{Formalization of the theorem}
As well as giving a (human-written) proof for the main theorem, we additionally provide a formalized proof by Lean 4. 
This formalization is publicly available on our GitHub repository \cite{K2024GitHUB}. We describe the outline in Section~\ref{sec:lean4}.
%
%We formalize the above concepts in Lean 4, a modern theorem prover and programming language based on dependent type theory. Lean 4 is designed to formalize the mathematical proofs.
Lean 4~\cite{Moura2021} is a modern theorem prover and programming language based on dependent type theory, and is designed to formalize mathematical proofs.

%We have formalized the proof of the main theorem, which states that the normalized degree sum of an ideal family is non-positive, using Lean 4. This formalization is publicly available on our GitHub repository \cite{K2024GitHUB}. 

The formal verification of the proof using a theorem prover like Lean 4 not only enhances the rigor of the result but also demonstrates the utility of formal methods in ensuring correctness in mathematical research. The ability to formalize and verify proofs provides a high level of assurance that can be difficult to achieve through traditional human verification.
This can benefit readers, for example, by making it easier to verify the proofs.
Formal proofs eliminate ambiguity and ensure that all logical steps are valid. This level of rigor is especially valuable in complex combinatorial proofs, where subtle errors can easily go unnoticed. In this paper, the correctness of the statements and their proofs have been rigorously validated using the Lean 4 system. 
% As a result, 
For verification,
it is not necessary for the readers to directly examine the Lean 4 proof code. Barring exceptional circumstances, it is reasonable to trust in the validity of the proofs. 
% Consequently, formal verification significantly reduces the effort required by the readers to confirm the correctness of the theorems, allowing them to focus on evaluating the relevance and usefulness of the results presented in the paper.
Formal proofs will reduce the efforts of the readers in verifying the correctness and, instead, the readers can focus their efforts in grasping the outlines.

The theorem prover can also identify gaps or errors that might be overlooked in human-written proofs. This capability helped us refine the inductive argument and ensure that all cases were properly handled. 
Lean 4 can detect conditions that are not used in the proof.

Once the definitions and lemmas are formalized in Lean 4,  they can be reused in future proofs and research. The structures and operations defined for ideal families can be extended to broader classes of set families, making Lean 4 a powerful tool for ongoing combinatorial research. It is also well-suited for collaborative proof development by large groups.

The study \cite{gusakov2024formalizing} formalizes the minor excluded characterization of binary matroids using the Lean 4 theorem prover. 
This study is not directly related to Frankl's conjecture, 
% but particularly in areas related to set families. 
but is relevant within the framework of set families.
We incorporated some of the definitions in his paper.

%\subsection{How to Create Lean 4 Code}
%\subsection{Methodorogy for Creating Lean 4 Code}
\bigskip
Here, we briefly explain how we created the proofs written in Lean 4.
% Historically, 
In general,
creating formal proofs in Lean 4 was cumbersome and time-consuming. However, the advent of AI assistants and large language models (LLMs) has introduced tools that greatly streamline high-level code creation. Thanks to these advancements, the process of writing proof code has become much more efficient. In creating the proofs presented here, we greatly benefited from the assistance of AI assistant tools such as ChatGPT~\cite{chatgpt}. 
There are several other tools that are helpful.
Lean Copilot \cite{lean_copilot} is a tool integrated into the Lean 4 system. GitHub Copilot~\cite{github_copilot} is an AI-powered coding assistant that supports various other programming languages. Lean search~\cite{leansearch} is a search engine for Lean 4 tactics. 
%Creating a proof in Lean 4 for our theorem was quite challenging and took several months, even with the help of these tools. 
However, writing formal proofs in Lean 4 is still a time-consuming work.
Currently, AI assistants often suggest outdated syntax or theorems in the format of Lean 3, which requires corrections and adjustments.  This issue is expected to be improved in the future. In the future, these tools are expected to evolve further and come closer to achieving automated theorem proving.

% \subsection{Organization of the Paper}

\bigskip\bigskip

The remainder of this paper is organized as follows. In Section 2, we provide definitions of key concepts, including set families, ideal families, normalized degree sum, and minors. Section 3 presents examples of ideal families to illustrate these concepts in a more concrete way and calculates their normalized degree sum. In Section 4, we detail the proof of the main theorem, that is, the average rarity condition for ideal families. Section 5 presents the formal proof of the main theorem by Lean 4. Section 6 concludes the paper and suggests directions for future research.

%\section{Mathematical Preliminaries}
\section{Mathematical preliminaries}

\subsection{Intersection-closed families, rarity, and average rarity}

Let $U$ be a nonempty finite ground set. An element of $U$ is called a \emph{vertex}. A set family $\mathcal{F}$ is a collection of subsets of $U$. An element of $\mathcal{F}$ is called a \emph{hyperedge}.  

\begin{definition}[Intersection-closed family]
A set family $\mathcal{F}$ on $U$ is \emph{intersection-closed} if, for any $A, B \in \mathcal{F}$, their intersection $A \cap B$ is also in $\mathcal{F}$. 
\end{definition}

\begin{definition}[Degree of a vertex]
For a vertex $v \in U$, the \emph{degree} of $v$ in $\mathcal{F}$, denoted $\deg_{\mathcal{F}}(v)$, is the number of hyperedges in $\mathcal{F}$ that contain the vertex $v$:

\[
\deg_{\mathcal{F}}(v) = |\{ H \in \mathcal{F} \mid v \in H \}|.
\]
\end{definition}

\begin{definition}[Rare vertex]
For a set family $\mathcal{F}$ and a vertex $v \in U$, $v$ is rare if $\deg_{\mathcal{F}}(v) \leq |\mathcal{F}|/2$.
\end{definition}

% The following is Frankl's conjecture.
By these terminologies, Frankl's conjecture can be expressed as follows.

\begin{conjecture}
Every intersection-closed set family $\mathcal{F}$ with $\{U,\emptyset\}\subseteq \mathcal{F}$ has a rare vertex.
\end{conjecture}

% \subsection{Normalized Degree Sum and Average Rarity}

In this paper, our main concern is in average rarity.
We say that a family is \emph{average rare} if the average of the degrees over all vertices is at most half the number of hyperedges of the family. This can be expressed using the normalized degree sum as follows.

\begin{definition}[Normalized degree sum]
The \emph{normalized degree sum} of $\mathcal{F}$ is defined as:

\[
{\rm NDS}(\mathcal{F}) = 2 \cdot \sum_{v \in U} \deg_{\mathcal{F}}(v) -  |U|\cdot |\mathcal{F}|.
\]

If ${\rm NDS}(\mathcal{F}) \leq 0$, we say that $\mathcal{F}$ satisfies the \emph{average rarity condition}.
\end{definition}

A family is average rare if it satisfies the average rarity condition, since $\text{NDS}(\mathcal{F}) \leq 0$ is equivalent to $\sum_{v\in U}\deg_{\mathcal{F}}(v) / |U|\le \frac{1}{2} |\mathcal{F}|$.
From this, the following lemma is straightforward.

\begin{lemma}\label{lemma:averagerare}
If a set family $\mathcal{F}$ satisfies the average rarity condition, $\mathcal{F}$ has a rare vertex.
\end{lemma}

Some intersection-closed set families may not satisfy the average rarity condition.
For example, $\{\,\emptyset, \{a\}, \{a,b\}, \{a,c\}, \{a,b,c\} \,\}$ is an intersection-closed family but ${\rm NDS}=1 > 0$. (Though  this family is not average rare, this family has rare vertices $b$ and $c$.)

The normalized degree sum measures the deviation from an even distribution of vertices among hyperedges. 
%A non-positive value indicates that, on average, vertices are not overly represented, which aligns with the concept of average rarity.

\medskip
Denote the total sum of the size of the hyperedge by 
$$\text{TSH}(\mathcal F) = \sum_{H\in {\mathcal F}}|H| = \sum_{v \in U} \deg_{\mathcal{F}}(v).$$
The second equality above follows from the double counting principle.
%Denote the number of hyperedges by NH$(\mathcal F) = |\mathcal{F}|$. Then, we have $\text{NDS}(\mathcal{F})=2\text{TSH}(\mathcal F)-\text{NH}(\mathcal F)\times|U|$.
By this, we have
$$ {\rm NDS}(\mathcal{F}) = 2\cdot {\rm TSH}(\mathcal{F}) - |U|\cdot |\mathcal{F}|.$$

\subsection{Ideal families}
\label{subsec:idealfamily}

\begin{definition}[Ideal family]
A set family $\mathcal{F} \subseteq 2^{U}$ is called an \emph{ideal family} if it satisfies the following conditions:

\begin{enumerate}
    \item \textbf{Contains Empty Set}: $\emptyset \in \mathcal{F}$.
    \item \textbf{Contains Ground Set}: $U \in \mathcal{F}$.
    \item \textbf{Downward-Closed except $U$}: For all $A, B \in \mathcal{F}$ with $A \neq U$, if $B \subseteq A$, then $B \in \mathcal{F}$.
\end{enumerate}
\end{definition}

In other words, an ideal family includes all subsets of its nonempty hyperedges, except the ground set. This property ensures that once a hyperedge (which is not the ground set) is contained, all smaller hyperedges contained by it are also included. 
%The existence of the empty set excludes the set family that has only the ground set. 
%This structure is crucial for analyzing and understanding properties such as average rarity. 
Obviously, every ideal family is closed under intersections.
Remark that the set family in which the ground set is its sole hyperedge is not allowed by the condition that the family must contain the empty set.

The next lemma can be proven relatively easily by demonstrating the existence of an injection, which is a commonly used argument in this field.
\begin{lemma} \label{lemma:idealrare}
Every ideal family has a rare vertex.
\end{lemma}
\begin{proof}
%In the case of the power set, all vertices are rare. If the set is not the power set, there exists a subset that is maximal under the inclusion relation, excluding the ground set. 
Consider a maximal hyperedge with respect to the inclusion relation, excluding the ground set. Let \( v \) be a vertex that does not belong to this maximal set.
To show that vertex \( v \) is rare, it suffices to construct an injection from hyperedges containing \( v \) to hyperedges not containing \( v \). Such a required mapping can be defined as follows:

\begin{itemize}
    \item If a hyperedge \( H \) containing \( v \) is not the ground set, then map \( H \) to the hyperedge \( H \setminus \{v\} \), which does not contain \( v \).
    \item If a hyperedge containing \( v \) is the ground set, map it to a maximal hyperedge under the inclusion relation that does not include \( v \), excluding the ground set.
\end{itemize}

This mapping is injective because different hyperedges are mapped to different hyperedges. If two hyperedges are mapped to the same hyperedge \( J \), the original hyperedge is  \( J \cup \{v\} \) (the case that \( J \cup \{v\} \) is a hyperedge) or the ground set (the case that \( J \cup \{v\} \) is not a hyperedge).

\end{proof}

The main theorem (Theorem~\ref{theorem:main}) of this
paper states that 
any ideal family $\mathcal{F}$ over a nonempty finite ground set $U$ is average rare, i.e., the normalized degree sum $\text{NDS}(\mathcal{F})\leq 0$.
That is, our main aim is to strengthen Lemma~\ref{lemma:idealrare} to average rarity.
Note that our proof of the main theorem relies on Lemma~\ref{lemma:idealrare}.

% We use this lemma to prove the main theorem.
%
% The statement is the main theorem is the following, which is proved later.
% For any ideal family $\mathcal{F}$ over a nonempty finite ground set $U$, it is average rare, that is, the normalized degree sum is non-positive: $\text{NDS}(\mathcal{F}) \leq 0$.

\subsection{Minors of ideal families}

Consider a set family $\mathcal{F}$ on the ground set $U$.  In this subsection, we assume that $\mathcal{F}$ has a ground set whose size is at least two.
We consider three operators, deletion, contraction, and trace. These operators map a set family to a set family whose ground set is smaller by one. These minor operations are defined for general set families, and especially, they preserve intersection-closedness~\cite{Hachimori2024}, that is, intersection-closed families are mapped to intersection-closed families by these operators.
In using these operators for ideal families, however, 
we need specialized treatments in order to ensure the operations preserve the families to be ideal.

% \begin{itemize}
    %\item \emph{Deletion minor}: $\del{F}{v}$
\paragraph{\emph{Deletion minor}: $\del{F}{v}$ \\[1mm]}
    
For a set family $\mathcal{F}$ and $v\in U$, 
the deletion $\del{F}{v}$ is the family consisting of
the hyperedges in $\mathcal{F}$ that do not contain $v$, i.e., $\del{F}{v} = \{ H \in \mathcal{F} \mid v \notin H \}$.
When $\mathcal{F}$ is an ideal family, the deletion is downward-closed, but it may not contain the ground set  
% F.ground \setminus \{v\}$. 
$U \setminus\{v\}$.
In this case, in order to make the family to be an ideal family, we add the ground set to the family. We define this operation as the deletion operator for ideal families as follows. 
$$\delp{F}{v} = \{ H \in \mathcal{F} \mid v \notin H \}\cup \{U\setminus \{v\}\}.$$
It is easy to verify that this family is an ideal family.

\begin{lemma} \label{lemma:deletion}
For an ideal family $\mathcal{F}$ and $v\in U$,  the deletion $\delp{F}{v}$ of $v$ 
%adding a hyperedge $U\setminus \{v\}$ 
is also an ideal family on the ground set $U \setminus \{v\}$.
\end{lemma}
    
% \item \emph{Contraction minor}: $\cont{F}{v}$
\paragraph{\emph{Contraction minor}: $\cont{F}{v}$ \\[1mm]}
For a set family $\mathcal{F}$ and $v\in U$, the contraction of $\mathcal{F}$ by $v$ is the collection of  all hyperedges containing $v$ in $\mathcal{F}$ and remove $v$ from each hyperedge, i.e., 
$$\cont{F}{v} = \{ H \setminus \{v\} \mid v \in H, H \in \mathcal{F} \}.$$
When $\mathcal{F}$ is an ideal family, $\cont{F}{v}$ is an ideal family under the condition that $\{v\}$ is a hyperedge.

\begin{lemma}
For an ideal family $\mathcal{F}$ and $v\in U$ with a hyperedge $\{v\}$,  the contraction $\cont{F}{v}$ of $v$ is an ideal family on the ground set $U \setminus \{v\}$.
\end{lemma}

\begin{proof}
Since $\{v\}$ is a hyperedge, the contraction minor has the empty set. 
\end{proof}

%In the proof of the main theorem, we consider cases based on whether $\{v\}$ is a hyperedge or not. 
The following lemma characterizes whether $\{v\}$ is a hyperedge or not, which will be used in the proof of the main theorem.

\begin{lemma}\label{lem:degone}
For an ideal family with $|U|\geq 2$, $\{v\}$ is not a hyperedge if and only if $\mbox{deg}_{\mathcal{F}} v=1$. 
\end{lemma}

\begin{proof}
Recall that the ground set is always contained in the ideal family.
When deg $v=1$, the ground set is the unique hyperedge which contains $v$.
\end{proof}

%{\color{red} exceptional cases}

% \item \emph{Trace minor}: $\trace{F}{v}$
\paragraph{\emph{Trace minor}: $\trace{F}{v}$ \\[1mm]}

For a set family $\mathcal{F}$ and $v\in U$, the trace $\trace{F}{v}$ of $\mathcal{F}$ by $v$ is the set family with hyperedges $\{ H \setminus \{v\} \mid H \in \mathcal{F} \}$. It is easy to show that the trace of an ideal family is always an ideal family.

\begin{lemma}\label{lemma:trace}
For an ideal family $\mathcal{F}$ and $v\in U$,  the trace $\trace{F}{v}$ by $v$ is an ideal family on the ground set $U \setminus \{v\}$.
\end{lemma}

% In our context, the trace minor only arises when the degree of $v$ is 1. In this case, it coincides with the deletion for an ideal family, and this property is utilized in the proof within Lean 4.

The trace minor is important in the discussions of intersection-closed families and ideal families, but will not appear explicitly in the proofs in this paper.
(When the degree of $v$ is 1, the trace minor coincides with the deletion for an ideal family, and such a case will appear in the proof of the main theorem.)

% \end{itemize}

\section{Examples}

To make the theoretical concepts more concrete, we provide examples of ideal families and calculate their normalized degree sums.

\subsection{Ground set with two vertices}
Let $U = \{v_1, v_2\}$. Consider the ideal family $\mathcal{F}$ defined as:

\[
\mathcal{F} = \{\emptyset, \{v_1\}, \{v_2\}, \{v_1, v_2\}\}.
\]

This family includes all subsets of $U$, satisfying the properties of an ideal family.
% The number of hyperedges is $|\mathcal{F}|=4$.
We have $|U|=2$ and $|\mathcal{F}|=4$.

The degrees are calculated as follows:
\[
\deg_{\mathcal{F}}(v_1) = 2 \quad \quad (\{v_1\}, \{v_1, v_2\}),
\]
\[
\deg_{\mathcal{F}}(v_2) = 2 \quad \quad (\{v_2\}, \{v_1, v_2\}).
\]

The total size of hyperedges is:
\[
\text{TSH}(\mathcal{F}) = \sum_{v \in U} \deg_{\mathcal{F}}(v) = 2 + 2 = 4.
\]

% The number of hyperedges is:
% \[
% |\mathcal{F}| = 4.
% \]

Hence, the normalized degree sum is:
\[
\text{NDS}(\mathcal{F}) = 2 \text{TSH}(\mathcal{F}) - |U|\cdot |\mathcal{F}|=  2 \cdot 4 - 2 \cdot 4 = 8 - 8 = 0.
\]

Therefore, $\mathcal{F}$ meets the average rarity condition. It is easy to verify that the normalized degree sum of the power set of a finite set is always 0.

All other ideal families with the ground set of size 2 also have non-positive NDS. 
%This fact is used as base case of the inductive proof of the main theorem.

\subsection{Ground set with three vertices}
Let $U = \{v_1, v_2, v_3\}$. Define the ideal family $\mathcal{F}$ as:

\[
\mathcal{F} = \{\emptyset, \{v_1\}, \{v_2\}, \{v_3\}, \{v_1, v_2\}, \{v_1, v_3\}, \{v_1, v_2, v_3\}\}.
\]
% The number of hyperedges is $|\mathcal{F}| = 7$.
We have $|U|=3$ and $|\mathcal{F}|=7$.

The degrees are calculated as follows:
\begin{align*}
&\deg_{\mathcal{F}}(v_1) = 4 \qquad (\{v_1\}, \{v_1, v_2\}, \{v_1, v_3\}, \{v_1, v_2, v_3\}),
\\
&\deg_{\mathcal{F}}(v_2) = 3 \qquad (\{v_2\}, \{v_1, v_2\}, \{v_1, v_2, v_3\}),
\\
&\deg_{\mathcal{F}}(v_3) = 3 \qquad (\{v_3\}, \{v_1, v_3\}, \{v_1, v_2, v_3\}).
\end{align*}

The total size of hyperedges is:
\[
\text{TSH}(\mathcal{F}) = \sum_{v \in U} \deg_{\mathcal{F}}(v) = 4 + 3 + 3 = 10.
\]

% The number of hyperedges:
% \[
% |\mathcal{F}| = 7.
% \]

Hence, the normalized degree sum is: 
\[
\text{NDS}(\mathcal{F}) = 2 \text{TSH}(\mathcal{F}) - |U|\cdot |\mathcal{F}|= 2 \cdot 10 - 3 \cdot 7 = 20 - 21 = -1.
\]

Thus, $\mathcal{F}$ satisfies the average rarity condition.

%\subsection{Ideal Family with a Vertex of Degree 1}\label{subsec:degreeone}
\subsection{Ideal family with a vertex of degree 1 and a large hyperedge}\label{subsec:degreeone}
We consider an ideal family with a vertex $v$ of degree 1 and the hyperedge $U\backslash \{v\}$. For an ideal family, $\{v\}$ is not a hyperedge when $v$ has degree 1 (Lemma \ref{lem:degone}).
%Consider a family of sets on a finite ground set \( U \), where, for a fixed \( v \in U \), the family consists of all sets that do not contain \( v \), together with the ground set \( U \). 
The hyperedges of this ideal family are given by \( \{ H \mid v \notin H \} \cup \{ U \} \).
This ideal family appears in one of the cases in the proof of the main theorem.

Let $|U| = n$. 
The number of hyperedges in this family is:
\[
|\mathcal{F}| = 2^{n-1} + 1.
\]

The total sum of the sizes of the hyperedges is calculated as:
\[
\text{TSH}(\mathcal{F}) = \bigg(\sum_{i=0}^{n-1} \binom{n-1}{i}\cdot i \bigg) + n
= (n-1) \cdot 2^{n-2} + n.
\]

The normalized degree sum is calculated as:
\[
\text{NDS}(\mathcal{F})=2 \text{TSH}(\mathcal{F}) - |U|\cdot |\mathcal{F}| 
= 2 ( (n-1)\cdot 2^{n-2}+n ) - n\cdot (2^{n-1}+1) =
n - 2^{n-1}.
\]

This value is non-positive for any natural number $n\geq 1$, hence the family satisfies the average rarity condition.

% \section{Proof of Average Rarity in Ideal Families}
\section{The average rarity of ideal families} \label{sec:main}

% \subsection{Statement of the Main Theorem}

The following is our main theorem.

\begin{theorem}\label{theorem:main}
For any ideal family $\mathcal{F}$ on a nonempty finite ground set $U$, the normalized degree sum is non-positive:

\[
\text{NDS}(\mathcal{F}) \leq 0.
\]

Consequently, all ideal families satisfy the average rarity condition.
\end{theorem}

\begin{proof}
% \subsection{Proof Outline}
We prove the theorem using induction on the size of the ground set $|U|$.
\medskip

% \textbf{Base Case}: For $|U| = 1$, the size of the ground set is one. The only hyperedges in the ideal family are the empty set and the ground set, hence, the normalized degree sum is zero.

The base case is when $|U|=1$.
In this case, the hyperedges in the ideal family are $\emptyset$ and $U$. This yields that the normalized degree sum is zero.
\medskip
% \textbf{Inductive Step}: Assume that the theorem holds for any ideal family whose ground set of size of $n-1$. We will prove the statement for a ground set $U$ with $|U| = n$.

For the inductive step, assume that the theorem holds for any ideal family with a ground set of size $n-1$.
We verify the statement for an ideal family $\mathcal{F}$ on $U$ with $|U|=n\ge 2$.
% \medskip

% \subsection{Proof of Inductive Step}
% Let $\mathcal{F}$ be an ideal family on the ground set $U$ with $|U| = n$.

% \[\text{NDS}(\mathcal{F}) = 2\text{TSH}(\mathcal{F})-n|\mathcal{F}|.\]

% \textbf{Step 1: Choose a vertex}

% Select an arbitrary rare vertex $v \in U$ since there exists a rare vertex by Lemma \ref{lemma:idealrare}.
By Lemma \ref{lemma:idealrare} there exists at least one rare vertex in $\mathcal{F}$, and choose one arbitrary rare vertex $v$.
%
% When $v$ is a rare vertex, we have:
Since $v$ is rare, we have:

\[
2  \deg_{\mathcal{F}}(v) - |\mathcal{F}|  \leq 0.
\]

% \textbf{Step 2: Case with a vertex of degree 1}
\medskip

% Distinguishing cases based on whether deg $v = 1$. Recall that when $\{v\}$ is not a hyperedge, contraction $\cont{F}{v}$ is not an ideal family because the minor does not have the empty set. 
% Here, we consider the case with deg $v$ = 1.
% When the ideal family has a vertex of degree one, the total size of hyperedges tends to be small. 

\noindent
\textbf{[The case $\deg_{\mathcal{F}}(v)=1$]} \\
First, we consider the case with $\deg_{\mathcal{F}}(v) = 1$.
The next lemma follows from Lemma \ref{lem:degone}.
\begin{lemma}
For an ideal family $\mathcal{F}$ with $|U| \geq 2$ and $v\in U$, $\{v\}$ is a hyperedge if and only if deg $v\geq 2$. 
\end{lemma}

%\begin{proof}
%It follows from that the ideal family satisfies the downward-closedness and the ground set is a hyperedge.
%\end{proof}

% In addition to the condition that $v$ has degree one, when $U\setminus \{v\}$ is a hyperedge in $\mathcal F$, see the example in Subsection \ref{subsec:degreeone}.
When $U\setminus \{v\}$ is a hyperedge in $\mathcal{F}$, the calculation in Subsection~\ref{subsec:degreeone} shows that ${\rm NDS}(\mathcal{F})\le 0$ and we are done.

When $U\setminus \{v\}$ is not a hyperedge in $\mathcal F$, 
% by the definition of trace minor we have
by the definition of the deletion operation,
%\[\text{TSH}(\mathcal{F}) =\text{TSH}(\trace{F}{v})+ 1,\]
\[\text{TSH}(\mathcal{F}) =\text{TSH}(\delp{F}{v})+ 1,\]and
% \[|\mathcal{F}| =|\trace{F}{v}|.\]
\[|\mathcal{F}| =|\delp{F}{v}|.\]
% (Remark that $\delp{F}{v}=\trace{F}{v}$ in this case.)

% These inequalities follow from the fact that the minor $\trace{F}{v}$ is obtained by replacing $F.ground$ by $\cal F$.ground $\setminus \{v\}$ from $\cal F$.
% The ground set $(\trace{F}{v}).$ground of the deletion $\trace{F}{v}$ is $\cal F$.ground $\setminus \{v\}=U \setminus \{v\}$.
% The minor $\trace{F}{v}$ is equal to be $\delp{F}{v}$.

% Since it has the empty set and the ground set $U\setminus \{v\}$ as hyperedges, $\trace{F}{v}$ is an ideal family. 
% Since a trace minor of an ideal family is an ideal family by Lemma~\ref{lemma:trace},
By Lemma~\ref{lemma:deletion}, $\delp{F}{v}$  is an ideal family on $U\setminus \{v\}$.
% $\trace{F}{v}$ is an ideal family on
% $U\setminus \{v\}$.
Therefore, 
% $$\text{NDS}(\trace{F}{v})= 2\text{TSH}(\trace{F}{v}) - (n-1)|\trace{F}{v}| \leq 0$$ 
$$\text{NDS}(\delp{F}{v})= 2\text{TSH}(\delp{F}{v}) - (n-1)|\delp{F}{v}| \leq 0$$ by the induction hypothesis.
Hence, we have
\begin{eqnarray*}
% \text{NDS}(\mathcal{F}) &=& 2\text{TSH}(\mathcal{F})-n|\mathcal{F}|\\
% &=&2(\text{TSH}(\trace{F}{v})+ 1)-n|\trace{F}{v}|\\
% &=& 2\text{TSH}(\trace{F}{v}) - (n-1)|\trace{F}{v}| + 2 - |\trace{F}{v}|\\
% &=& \text{NDS}(\trace{F}{v}) + (2 - |\trace{F}{v}|)\\
% &\leq& 0.
\text{NDS}(\mathcal{F}) &=& 2\text{TSH}(\mathcal{F})-n|\mathcal{F}|\\
&=&2(\text{TSH}(\delp{F}{v})+ 1)-n|\delp{F}{v}|\\
&=& 2\text{TSH}(\delp{F}{v}) - (n-1)|\delp{F}{v}| + 2 - |\delp{F}{v}|\\
&=& \text{NDS}(\delp{F}{v}) + (2 - |\delp{F}{v}|)\\
&\leq& 0.
\end{eqnarray*}
%
% Note that $|\trace{F}{v}|\geq 2$ holds because it contains the empty set and the ground set.
Note that $|\delp{F}{v}|\geq 2$ holds because it contains the empty set and the ground set.

% Therefore, the case deg $v = 1$ is settled by the inductive assumption. We assume that $\{v\}$ is a hyperedge in the following.
We conclude that ${\rm NDS}(\mathcal{F})\le 0$ in the case $\deg_{\mathcal{F}}(v)=1$.

% \textbf{Step 3: Distinguishing the case according to whether $U\setminus \{v\}$ is a hyperedge or not}

\bigskip
\noindent
\textbf{[The case $\deg_{\mathcal{F}}(v)\ge 2$]}

We consider the case 
$\deg_{\mathcal{F}}(v)\ge 2$, i.e., the case that
$\{v\}$ is a hyperedge.
%In addition, assume that $U\setminus \{v\}$ is a hyperedge.
%When $U\setminus \{v\}$ is a hyperedge, the deletion of an ideal family is slightly changed.
%Therefore, the proof of the non-positivity of normalized degree sum is a bit different. Basically, the case where $U\setminus \{v\}$ is a hyperedge is hard to hold inequalities.

% According to the induction hypothesis, both $\delp{F}{v}$ and $\cont{F}{v}$ satisfy the average rarity condition:
Since $\delp{F}{v}$ and $\cont{F}{v}$ are ideal families by Lemmas~\ref{lemma:deletion} and \ref{lemma:trace}, they satisfy the average rarity condition:
\[
\text{NDS}(\delp{F}{v}) \leq 0, \quad \text{NDS}(\cont{F}{v}) \leq 0.
\]

By definitions, we have the following.

\[\text{NDS}(\cont{F}{v}) = 2\text{TSH}(\cont{F}{v})-(n-1)|\cont{F}{v}|,\]
\[\text{NDS}(\delp{F}{v}) = 2\text{TSH}(\delp{F}{v})-(n-1)|\delp{F}{v}|,\]
\[|\mathcal{F}| =|\cont{F}{v}|+|\del{F}{v}|,\]
\[\text{TSH}(\mathcal{F}) =\text{TSH}(\cont{F}{v})+\text{TSH}(\del{F}{v})+\text{deg}_{\mathcal F} v.\]

We divide the cases depending on whether $U\setminus\{v\}$ is a hyperedge of $\mathcal{F}$ or not.
\begin{itemize}
\item In the case that $U\setminus \{v\}$ is a hyperedge of $\mathcal{F}$, we have
\[|\del{F}{v}|=|\delp{F}{v}|,\]
and
\[\text{TSH}(\del{F}{v})=\text{TSH}(\delp{F}{v}).\]

% We observe the following relationship between the normalized degree sums:
The statement is deduced as follows.

\begin{eqnarray*}
&&\text{NDS}({\mathcal F}) \\
&=& 2\text{TSH}({\mathcal F}) - n|\mathcal{F}|\\
&=& 2(\text{TSH}(\cont{F}{v})+\text{TSH}(\del{F}{v})+\text{deg}_{\mathcal F}) - n(|\cont{F}{v}|+|\del{F}{v}|)\\
&=& (2\text{TSH}(\cont{F}{v})-(n-1)|\cont{F}{v}|)+(2\text{TSH}(\del{F}{v})-(n-1)|\del{F}{v}|)\\
&&+  2  \deg_{\mathcal{F}}(v) -(|\cont{F}{v}|+|\del{F}{v}|) \\
&=& \text{NDS}(\delp{F}{v}) + \text{NDS}(\cont{F}{v}) + 2 \deg_{\mathcal{F}}(v) - |\mathcal{F}|\\
&\leq& 0.
\end{eqnarray*}

% This equation and inequality account for the number of hyperedges in $\mathcal{F}$, $\del{F}{v}$, and $\cont{F}{v}$ and the rarity of $v$.
The last inequality follows from $2 \deg_{\mathcal{F}}(v) - |\mathcal{F}|\leq 0$.

%Combining the inequalities:
%\[
%\text{NDS}(\mathcal{F}) \leq \text{NDS}(\del{F}{v}) + \text{NDS}(\cont{F}{v})) \leq 0.
%\]

\item If $U\setminus \{v\}$ is not a hyperedge of $\mathcal{F}$, we have
\[|\del{F}{v}|=|\delp{F}{v}| - 1,\]
and
\[\text{TSH}(\del{F}{v})=\text{TSH}(\delp{F}{v}) - n + 1.\]
Hence,
\begin{eqnarray*}
\text{NDS}(\del{F}{v}) &=& 2{\rm TSH}(\del{F}{v})-(n-1)|\del{F}{v}|\\
&=&2\text{TSH}(\delp{F}{v})-(n-1)|\delp{F}{v}| -2n+2+(n-1)\\
&=&\text{NDS}(\delp{F}{v})-n+1.
\end{eqnarray*}

% Using a similar calculation as above, we find that $\mathcal{F}$ satisfies the average rarity condition.
The statement is verified as follows.
\begin{eqnarray*}
&&\text{NDS}({\mathcal F}) \\
&=& (2\text{TSH}(\cont{F}{v})-(n-1)|\cont{F}{v}|)+(2\text{TSH}(\del{F}{v})-(n-1)|\del{F}{v}|)\\
&&+  2  \deg_{\mathcal{F}}(v) -(|\cont{F}{v}|+|\del{F}{v}|) \\
&=&(2\text{TSH}(\cont{F}{v})-(n-1)|\cont{F}{v}|)\\
&&+(2(\text{TSH}(\delp{F}{v})-n+1)-(n-1)(|\delp{F}{v}|-1)\\
&&+  2  \deg_{\mathcal{F}}(v) -(|\cont{F}{v}|+|\del{F}{v}|)\\
&=& \text{NDS}(\delp{F}{v}) + \text{NDS}(\cont{F}{v}) + (2 \deg_{\mathcal{F}}(v) - |\mathcal{F}|)-n+1\\
&\leq& 0.
\end{eqnarray*}

Recall that $2 \deg_{\mathcal{F}}(v) - |\mathcal{F}|\leq 0$ and $n \geq 2$.

\end{itemize}

\end{proof}

\noindent
\textit{Remark}: 
The proof relies on $\mathcal{F}$ being an ideal family to guarantee the existence of a rare vertex $v$.
In addition, it is used to ensure that the contraction $\cont{F}{v}$ includes the empty set.
For a general intersection-closed family, it may not contain the empty set.
In the base case, ${\rm NDS}\le 0$ is not guaranteed if the empty set is not contained in the family.

\section{Formalization in Lean 4} \label{sec:lean4}

In this section, we present a formalized proof of the main theorem in Lean 4. 
The code shown here is slightly simplified from the full implementation to enhance clarity, such as by excluding conversions between natural numbers and integers.
The complete implementation can be found in the repository~\cite{K2024GitHUB}.
%The ground set $U$ is written as $F.ground$ in our Lean code.

\subsection{Basic Definitions in Lean 4}

%The following are the basic definitions used in our framework.
Below, we outline the core definitions employed in our Lean 4 formalization.

\begin{lstlisting}[caption={Definition of SetFamily and IdealFamily in Lean 4}]
-- Definition of set families
structure SetFamily (§$\alpha$§ : Type) :=
  (ground : Finset §$\alpha$§)
  (sets : Finset §$\alpha$§ §$\rightarrow$§ Prop)
  (inc_ground : §$\forall$§ s, sets s §$\rightarrow$§ s §$\subseteq$§ ground)
  (nonempty_ground : ground.Nonempty)

-- Definition of intersection-closed families
def isIntersectionClosedFamily  (F : SetFamily §$\alpha$§) : Prop :=
  §$\forall$§ {s t : Finset §$\alpha$§}, F.sets s §$\rightarrow$§ F.sets t §$\rightarrow$§ F.sets (s §$\cap$§ t)

-- Definition of rare vertices
def is_rare (F : SetFamily §$\alpha$§) (v : §$\alpha$§) : Prop :=
  2 * F.degree v - F.number_of_hyperedges §$\leq$§ 0

-- Ideal families
structure IdealFamily  (§$\alpha$§ : Type) extends SetFamily §$\alpha$§ :=
  (has_empty : sets §$\emptyset$§)
  (has_ground : sets ground)
  (downward_closed : §$\forall$§ (A B : Finset §$\alpha$§), sets B §$\rightarrow$§ B §$\neq$§ ground §$\rightarrow$§ A §$\subseteq$§ B §$\rightarrow$§ sets A)

-- Total size of hyperedges
def SetFamily.total_size_of_hyperedges (F : SetFamily §$\alpha$§)   : §$\mathbf Z$§ :=
   ((Finset.powerset F.ground).filter F.sets).sum Finset.card

-- Number of hyperedges
def SetFamily.number_of_hyperedges  (F : SetFamily §$\alpha$§): §$\mathbf Z$§ :=
   ((Finset.powerset F.ground).filter F.sets).card

-- Normalized degree sum
def SetFamily.normalized_degree_sum (F : SetFamily §$\alpha$§) :=
  2 * F.total_size_of_hyperedges - F.number_of_hyperedges*F.ground.card
\end{lstlisting}

In this formalization:
 `\verb+SetFamily+' represents a general set family on  finite type `$\alpha$'.
 `\verb+IdealFamily+' extends `\verb+SetFamily+' with the specific properties of an ideal family. \\ %line-break inserted
 `\verb+downward_closed+' is a condition that hyperedges are closed with respect to taking subsets. 
 %In reality, the conversion between \verb+SetFamily+ and \verb+IdealFamily+ requires a function, but we omit it in the pseudo-code within this paper.
`\verb+Set.card+' means the cardinality of the set.
The ground set $U$ of a set family `\verb+F+' is written as `\verb+F.ground+'.

%The degree is defined as follows using pseudo lean 4 code:
%\begin{lstlisting}[caption={Definition of degree}]
%def degree
%\end{lstlisting}

The Lean 4 code below states Frankl's conjecture,
with `\verb+sorry+' indicating that the proof remains incomplete.

\begin{lstlisting}[caption={Frankl's conjecture}]
theorem frankl_conjecture : 
  §$\forall$§  (F : SetFamily §$\alpha$§):
    has_empty  F §$\rightarrow$§
    has_univ F §$\rightarrow$§
    is_closed_under_intersection F §$\rightarrow$§
    §$\exists (v \in$§ F.ground), 2 * F.degree v §$\leq$§ F.number_of_hyperedges := sorry
\end{lstlisting}

Here, \verb+F.degree v+ indicates the degree of a set family $\mathcal F$ at $v$.

\subsection{Basic structure of the proof of the main theorem}

%The normalized degree sum is represented as \verb+normalized_degree_sum+.

%\begin{lstlisting}[caption={Definition of normalized degree sum}]
%def normalized_degree_sum 
%  (F : SetFamily §$\alpha$§) : §$\mathbf Z$§ :=
%  2*(F.total_size_of_hyperedges) - (F.number_of_hyperedges)*F.ground.card
%\end{lstlisting}

\begin{lstlisting}[caption={Concept proof of the main theorem in Lean 4}]
theorem ideal_average_rarity (F : IdealFamily §$\alpha$§):
  F.normalized_degree_sum §$\leq 0 $§:= by
  -- Induction on the size of the ground set
  induction F.ground.card with
  | one =>
      exact nds_nonposi_card_one F
  | succ ih => -- ih is induction hypothesis.
    -- v is a rare vertex, and rv provides the evidence for it.
    obtain §$\langle v, rv \rangle$§ := ideal_version_of_frankl_conjecture F

    have geq2: F.ground.card §$ \geq 2 $§ := by sorry  
    -- proof for this is omitted here.

    by_cases h_v : F.sets {v}
    case pos =>
      -- Now consider whether (F.ground \ {v}) is a hyperedge
      by_cases h_uv : F.sets (F.ground \ {v})
      case pos =>
        -- If (U\{v}) is a hyperedge
        exact case_hs_haveUV F v h_v rv geq2 h_uv ih
      case neg =>
        -- If (U\{v}) is not a hyperedge
        exact case_hs_noneUV F v h_v rv geq2 h_uv ih
    case neg =>
      by_cases h_uv : F.sets (F.ground \ {v})
      case pos =>
        -- If (U\{v}) is a hyperedge
        exact case_degone_haveUV F v rv geq2 h_v h_uv
      case neg =>
        exact case_degone_noneUV F v rv geq2 h_v h_uv ih
\end{lstlisting}

The code above offers a conceptual overview rather than executable Lean 4 code. Still, readers versed in Lean 4 syntax can view it as a pseudocode capturing the proof's overall structure.

In this code:
The sentences followed by `\verb+--+' are comments.
 In theorem \\ %line-break inserted
 `\verb+ideal_average_rarity+ \verb+(F : IdealFamily+ $\alpha$\verb+)+ :
  \verb+normalized_degree_sum+ F $\leq$ 0', the part of `\verb+(F : IdealFamily+ $\alpha$\verb+)+' is an assumption, and the part of `\verb+normalized_degree_sum  F + $\leq$ {\tt 0}' is the goal to be proved. Using tactics in Lean 4, we can transform the goal step by step and then complete the proof by demonstrating that the goal follows from the assumptions, for example, using \verb+exact+.
The part following by `\verb+:=+' is the proof of the theorem.

The proof proceeds by induction of `\verb+F.ground+', the cardinality of the ground set of the ideal family.
For the base case, we handle the situation where the ground set is of size one.

In the inductive step, we select a vertex $v$ as a rare vertex by using theorem \\ %line-break inserted `\verb+ideal_version_of_frankl_conjecture+' which state that every ideal family has a rare vertex.
We define the deletion and contraction families `$\delp{F}{v}$' and `$\cont{F}{v}$'.
Under the appropriate assumptions, it is clear that these are ideal families. Since these minors have a smaller ground set, we can apply the induction hypothesis to both minors. 
Since the deletion minors and contraction minors differ depending on whether $\{v\}$ or $\text{F.ground} \setminus \{v\}$ is a hyperedge, it is necessary to consider separate cases.
To resolve subcases, we use lemma `\verb+case_hs_haveUV+' and others. 
The part following from the lemma name corresponds to the argument and is assumed to hold.

In the proof, \verb+cases+ represents case analysis. `\verb+case pos =>+' indicates the case where the condition holds, while `\verb+case neg =>+' represents the case where the condition does not hold.
`\verb+have label:statement+' is a mini lemma used in the proof.
`\verb+sorry+' represents  the omission of a proof.

The next code is a concise version of the statement of a key lemma when $U\setminus \{v\}$ and $\{v\}$ are both hyperedges. The part before the outer colon is the assumption of this lemma, and the part after the outer colon is the goal of this lemma.

The following lemma in Lean 4 corresponds to
$$\text{NDS}({\mathcal F}) = \text{NDS}(\del{F}{v}) + \text{NDS}(\cont{F}{v}) + 2 \deg_{\mathcal{F}}(v) - |\mathcal{F}|.$$

\begin{lstlisting}[caption={Concept statement of a key lemma in Lean 4}]
lemma nds_set_minors (F : IdealFamily §$\alpha$§)  (v : §$\alpha$§) (hv : v §$\in$§ F.ground) (geq2: F.ground.card §$\geq$§ 2)
 (hs : F.sets {v}):
  F.toSetFamily.normalized_degree_sum =
  (F.toSetFamily.deletion v hv geq2).normalized_degree_sum +
  (F.toSetFamily.contraction v hv geq2).normalized_degree_sum
  + 2 * (F.degree v) - F.number_of_hyperedges :=
\end{lstlisting}

The deletion $\delp{F}{v}$ for set families corresponds to  
``\verb+F.toSetFamily.deletion' v hv geq2+'' in Lean 4 code. The deletion $\delp{F}{v}$ for ideal families corresponds to `\verb+F.deletion v hv geq2+'.
The following lemma corresponds to $\text{NDS}(\del{F}{v}) =\text{NDS}(\delp{F}{v})-n+1$ under the assumption that $\verb +F.ground+ {\tt \setminus \{v\}}$ is not a hyperedge.

\begin{lstlisting}[caption={Relation between deletion and ideal deletion in Lean 4}]
lemma nds_deletion_noneuv (F : IdealFamily §$\alpha$§)  (v : §$\alpha$§) (hv : v §$\in$§ F.ground) (geq2: F.ground.card §$\geq$§ 2)
   (h_uv : §$\neg$§F.sets (F.ground \ {v})) :
  (F.deletion' v hv geq2).normalized_degree_sum = (F.toSetFamily.deletion v hv geq2).normalized_degree_sum + F.ground.card - 1 :=
\end{lstlisting}

% We relate the normalized degree sums using the established relationships.
We connect the normalized degree sums through the relationships defined earlier.
%Finally, we use tactic `\verb+linarith+' to handle the inequalities and conclude that `\verb +normalized_degree_sum+ {\tt F $\leq$ 0}'.
Then, by applying the Lean 4 tactic `\verb+linarith+' to resolve the linear inequalities, we establish that  `\verb +normalized_degree_sum+ {\tt F $\leq$ 0}'.
  
%The use of `noncomputable` reflects the fact that certain calculations may not terminate automatically due to the complexity of set operations.

The full formal proof of our result by Lean 4 is published in our GitHub repository \cite{K2024GitHUB}:
\url{https://github.com/kashiwabarakenji/}

\section{Concluding remarks}

We have shown that every ideal family has a non-positive normalized degree sum, meaning it satisfies the average rarity condition, using induction on the size of the ground set.
%This study improves our understanding of ideal families and provides a foundation for exploring similar properties in broader classes of set families. 
% The proof that ideal families satisfy the average rarity condition offers key insights into the structure of intersection-closed families and contributes to ongoing efforts to understand and potentially prove Frankl's conjecture.
This proof offers insights into understanding how intersection-closed families will meet the average rarity condition.
Future research will explore the applicability of the average rarity condition to broader classes of intersection-closed families
and further developments related to Frankl’s conjecture.
% and further develop formal proofs for theorems associated with Frankl’s conjecture.

% Additionally, the research highlights the value of formal methods, such as Lean 4, in ensuring mathematical rigor and reliability.
We also formalized the proof in Lean 4. 
As more theorems in this field are formalized in Lean 4 and made publicly available, 
%it will contribute to AI potentially learning from these formalizations and significantly improve its expertise in this domain.
this could enable AI to learn from them, enhancing its capabilities in combinatorial research.

By assuming that a set family has the ground set as one of its hyperedges, a set family closed under intersection can be represented as a closure system. Consequently, the set family can be expressed in terms of rooted circuits. The approach of considering Frankl's conjecture in the context of rooted circuits will be explored in a separate paper in the future.

%In this paper, the property that closed set families are average rare is proved using a lemma that guarantees the existence of a rare vertex. 
This paper establishes that ideal families are average rare by using a lemma ensuring a rare vertex exists.
%Generally, the advantage of considering the average rarity lies in the fact that the existence of a rare vertex can be deduced from the average rarity. 
Generally, exploring the average rarity proves advantageous because it implies the existence of a rare vertex. 
However, the approach taken in this paper is in a sense reversed: 
%instead of directly investigating whether a set family is average rare without relying on the lemma concerning rare vertices, the proof relies on the existence of a rare vertex. 
we deduced the average rarity of ideal families from the existence of a rare vertex.
% Whether it is possible to determine average rarity directly without using the existence of a rare vertex remains an open question.
% It remains an open question whether the average rarity can be determined directly without relying on the existence of a rare vertex.
Whether average rarity can be established independently of the existence of a rare vertex remains unresolved.

\bigskip

\subsection*{Funding Statement}
The authors have no relevant financial or non-financial interests to disclose.

\subsection*{Data Availability Statement}
All formalized proofs and related Lean 4 source code used in this article are publicly available at the following GitHub repository
URL and can be freely accessed and verified. 
\\[1mm]
\url{https://github.com/kashiwabarakenji/frankl_lean}
\\[1mm]
No proprietary or confidential data were used.

\end{document}